\documentclass[a4paper,11pt]{article}%
\usepackage{mathptmx} %
\DeclareMathAlphabet{\mathcal}{OMS}{cmsy}{m}{n}
\usepackage[latin1]{inputenc}
\usepackage[T1]{fontenc}
\usepackage[english]{babel}

\usepackage{stmaryrd} % for the symbol [| f |]_{\beta}
\usepackage{graphicx}

\usepackage{indentfirst}
\usepackage{tablefootnote}

\usepackage[allcolors=cyan,colorlinks=true, linkcolor=cyan]{hyperref}
\usepackage[lmargin=3cm,tmargin=2cm,bmargin=2cm,rmargin=3cm]{geometry}

\usepackage{bbm}
\usepackage{amsmath,amssymb, amsthm,amsfonts,mathtools} 
\usepackage{pifont}
\usepackage{times}
\renewcommand\qedsymbol{\ding{111}}
\usepackage{mathtools}

\def\XXint#1#2#3{{\setbox0=\hbox{$#1{#2#3}{\int}$}
		\vcenter{\hbox{$#2#3$}}\kern-.5\wd0}}

\newcommand{\fin}{\hfill \qedsymbol}
\newcommand{\Csharp}{{\settoheight{\dimen0}{C}\kern-.09em \resizebox{!}{\dimen0}{\raisebox{\depth}{$\sharp$}}}}

\newtheoremstyle{theorem}
{5pt +1\p@ -2.0\p@}% above space (default)
{5pt +1\p@ -2.0\p@}% below space
{\it}			      % Body font
{}				  % Indent amount
{\bfseries}   % Theorem head font
{.}               % Punctuation after theorem head
{.4em}       % Space after theorem head
{}               % Theorem head spec (can be left empty, meaning `normal')
\theoremstyle{theorem}
\newtheorem{theorem}{Theorem}[section]

\newtheorem{corollary}[theorem]{Corollary}
\newtheorem{lemma}[theorem]{Lemma}
\newtheorem{remark}[theorem]{Remark}

\numberwithin{equation}{section}

\begin{document}
\title{\Large{Adams' trace principle in Morrey-Lorentz spaces on $\beta$-Hausdorff dimensional surfaces}}
\author{{Marcelo F. de Almeida}{\thanks{%
			de Almeida, M.F. was supported by CNPq:409306/2016, Brazil (Corresponding author).}} \\
	%EndAName
	{\small Universidade Federal de Sergipe, Departamento de Matem\'atica,} \\
	{\small {\ CEP 49100-000, Aracaju-SE, Brazil.}}\\
	{\small \texttt{e-mail:marcelo@mat.ufs.br}}\vspace{0.5cm}\\
	{Lidiane S. M. Lima}{\thanks{%
			Lima, L.S.M. was supported by CNPq:409306/2016, Brazil.}}\\
	{\small Universidade Federal de Goi\'as, IME - Departamento de Matem\'atica,} 
	\\
	{\small CEP 740001-970, Goiania-GO, Brazil.} \\
	{\small \texttt{e-mail:lidianesantos@ufg.br}}}
\date{}
\maketitle
\begin{abstract}
In this paper we strengthen to Morrey-Lorentz spaces the famous trace principle introduced by Adams. More precisely, we  show that Riesz potential $I_{\alpha}$ is continuous 
 \begin{equation}
\Vert I_{\alpha}f\Vert_{\mathcal{M}_{q, \infty}^{\lambda_{\ast}}(d\mu)}\lesssim \Arrowvert\mu\Arrowvert_{\beta}^{{1}/{q}}\,\Vert f\Vert_{\mathcal{M}_{p, \infty}^{\lambda}(d\nu)}\nonumber\\[0.02in]
\end{equation}
if and only if the Radon measure $d\mu$ supported in $\Omega\subset \mathbb{R}^n$ 
is controlled by $$\llbracket\mu\rrbracket_{\beta}=\sup_{x\in\mathbb{R}^n,\,r>0}r^{-\beta}\mu(B(x,r))<\infty$$ 
provided that  $1<p<q<\infty$ satisfies $n-\alpha p<\beta\leq n,\; \alpha=\frac{n}{\lambda}-\frac{\beta}{\lambda_\ast}\; \text{ and }\;\frac{\lambda_\ast}{q}\leq  \frac{\lambda}{p}\nonumber\,$. Our result provide a new class of functions spaces which is larger than previous ones, since we have strict continuous inclusions 
$
\dot{B}_{p,\infty}^{s}\hookrightarrow
L^{\lambda, \infty}\hookrightarrow \mathcal{M}_{p}^{\lambda}\hookrightarrow\mathcal{M}_{p, \infty}^{\lambda} \nonumber
$
as $1<p<\lambda<\infty$ and $s\in\mathbb{R}$ satisfies $\frac{1}{p}-\frac{s}{n}=\frac{1}{\lambda}$. If $d\mu$ is concentrated on $\partial\mathbb{R}^n_+$, as a byproduct we get Sobolev-Morrey trace inequality on half-spaces $\mathbb{R}^n_+$ which recovers the well-known Sobolev-trace inequality in $L^p(\mathbb{R}^n_+)$. Also, by a suitable analysis on non-doubling Cader\'on-Zygmund decomposition we show that 
\begin{equation}
\Vert M_{\alpha}f\Vert_{\mathcal{M}_{p, \ell}^{\lambda}(d\mu)}\,\sim\, 	\Vert I_{\alpha}f\Vert_{\mathcal{M}_{p, \ell}^{\lambda}(d\mu)}\nonumber 
\end{equation}
provided that $\mu(B_r(x))\sim r^{\beta}$ on support $\text{spt}(\mu)$ and $n-\alpha <\beta\leq n$ with $0<\alpha<n$. This result extends the previous ones. 
%and $\Omega \supseteq\text{spt}(\mu)$ is a smooth surface with %$\vert \widehat{d\mu}(\xi)\vert \lesssim \vert \xi\vert^{-\frac{k}{2}}$ %over smooth surfaces $\Omega$ with 
%$k$ non-vanishing principal curvatures then tracing principle can be applied.


\end{abstract}

%\tableofcontents
\setcounter{equation}{0}\setcounter{theorem}{0}

\section{Introduction}
It is well known that doubling property of measures $\mu$ plays an important role in many topics of research in analysis on euclidean spaces, essentially because Vitali covering lemma and Calder\'on-Zygmund decomposition depend of the doubling property  $\mu(B_{2r}(x))\lesssim\mu(B_r(x))$ for all $x$ on support $\text{spt}(\mu)$ of measure $\mu$ and $r>0$. Recently it has been shown that fundamental results in analysis remain if doubling measures is replaced by a  \textit{growth condition}, namely, 
\begin{equation}\label{non-doubling}
\mu(B_r(x))\lesssim C\, r^{\beta}\;\; \text{ for all }\;\; x\in \text{spt}(\mu) \text{  and } \;\;r>0,
\end{equation}
where the implicit constant is independent of $\mu$ and $0<\beta\leq n$. For instance, we refer the pioneering work on Calder\'on-Zygmund theory for non-doubling measures \cite{Tolsa1,Tolsa, Tolsa2} and \cite{Nazarov}.  According to Frostman's lemma \cite[Chapter 1]{Mattila}, a measure $\mu$ satisfying \eqref{non-doubling} is close to Hausdorff measure and Riesz capacity of Borel sets  $\Omega\subset\mathbb{R}^n$. Essentially Frostman's lemma states that Hausdorff dimension of a Borel set $\Omega\subset\mathbb{R}^n$ is equal to 
\begin{equation}
\text{dim}_{\Lambda_{\beta}}\Omega=\sup\{\beta\in (0,n]: \exists\, \mu \in \mathcal{M}(\Omega) \text{ such that } \eqref{non-doubling} \text{ holds} \} =\sup \{\beta>0: \text{cap}_\beta(\Omega)>0\}\nonumber
\end{equation}
where $\Lambda_{\beta}(\Omega)$ denotes the $\beta-$dimensional Hausdorff measure and  $\text{cap}_\beta(\Omega)$ denotes the Riesz capacity,
\begin{equation}
\text{cap}_\beta(\Omega) = \sup\left\{ [\text{E}_{\beta}(\mu)]^{-1}\,:\, \text{E}_{\beta}(\mu) = \int_{\mathbb{R}^n} \int_{\mathbb{R}^n} \vert x-y\vert^{-\beta}d\mu(x)d\mu(y) \text{ for finite Borel measure } \mu \right\}\nonumber.
\end{equation}
The $L^p$-Riesz capacity on compact sets 
\begin{equation}
\dot{c}_{\alpha,p}(E)=\text{inf}\left\{\int_{\mathbb{R}^n}\vert f(x)\vert^p d\nu(x) \,:\, f\geq 0\text{ and } I_{\alpha}f(x)\geq 1 \text{ on } E\right\},\nonumber
\end{equation}
plays an important role in potential analysis, where $I_{\alpha}$ is defined  by 
\begin{align*}
I_{\alpha}f(x)=C_{\alpha,n}\int_{\mathbb{R}^n}\vert x-y\vert ^{\alpha-n}f(y)d\nu(y)\quad \text{ a.e. } x\in\mathbb{R}^n \quad \text{ as }\quad 0<\alpha<n
\end{align*}
where $d\nu$ stands for Lebesgue measure in $\mathbb{R}^n$. It is well known from  \cite[Theorem 7.2.1]{Adams-Hedberg} and \cite[Theorem 1]{Adams3,Dahlberg} that a necessary and sufficient condition for Sobolev embedding 
\begin{equation}
\dot{L}_{p}^{\alpha}(\mathbb{R}^n)\hookrightarrow  L^{q}(\Omega,\mu) \nonumber
\end{equation}
on the ``lower triangle'' $\,1<p\leq q<\infty$, $\,0<\alpha<n$ and $\,p<{n}/{\alpha}\,$ is given by the isocapacitary inequality
\begin{equation}\label{cap-cond}
\mu(E)\lesssim [\dot{c}_{\alpha,p}(E)]^{q/p},
\end{equation}
whenever $E$ is a compact subset of $\mathbb{R}^n$ and  $\mu$  is a Radon measure in  $\Omega$. Since $\dot{c}_{\alpha,p}(B_r(x))\cong r^{n-\alpha p}$, then \eqref{cap-cond} implies the \textit{growth condition} (\ref{non-doubling}) with $\beta= q(n/p-\alpha)$. Capacity  inequality is very difficult to verify even for compact sets, then one could ask: does the embedding $\dot{L}_{p}^{\alpha}(\mathbb{R}^n)\hookrightarrow L^{q}(\Omega,\mu)$ still hold if  (\ref{cap-cond}) is replaced by  (\ref{non-doubling})? %The pioneering work on characterization of  embedding  $\dot{L}_{\alpha}^{p}(\mathbb{R}^n)\hookrightarrow L^{q}(\Omega,\mu)$ by Radon measures $\mu$ satisfying  growth condition \eqref{non-doubling}, was first showed by Adams 
 In \cite[Theorem 2]{Adams0} Adams gave a positive answer to this question as $1<p<q<\infty$ and $\beta= q(n/p-\alpha)$ satisfies $0<\beta\leq n$ and $0<\alpha <n/p$. This theorem has a weak-Morrey version \cite[Theorem 5.1]{Adams} (see also  \cite[Lemma 2.1]{Xiao2}) and a strong Morrey version \cite[Theorem 1.1]{Liu}. Let us be more precise. The Morrey space $\mathcal{M}_r^\ell(\Omega,d\mu)$ is defined by the space of $\mu-$measurable functions $f\in L^r(\Omega\cap B_R)$ such that 
\begin{equation}
\Vert f \Vert_{\mathcal{M}^\ell_{r}(\Omega,d\mu)}=\sup_{x\,\in\, \text{spt}(\mu),\, R>0}R^{-\beta\left(\frac{1}{r}-\frac{1}{\ell}\right)} \left(\int_{B_{R}}\vert f(y)\vert^{r}d\mu{\lfloor_{\Omega}}\right)^{\frac{1}{r}}<\infty\nonumber,
\end{equation} 
where the supremum is taken on balls $B_R(x)\subset\mathbb{R}^n$, $1\leq r\leq \ell<\infty$ and $\beta>0$ denotes the Housdorff dimension of $\Omega$. In \cite{Liu} the space  $\mathcal{M}_r^\ell(d\mu)$ is denoted by $L^{r,\kappa}(d\mu)$ with  $\kappa/r=n/{\ell}$ and in \cite{Lucas-bilinear} is denoted by $\mathcal{M}_{r,\kappa}(d\mu)$ with $(n-\kappa)/r=n/\ell$. The Morrey-Lorentz space $\mathcal{M}^{\ell}_{r, s}(\Omega,d\mu)$ is defined by space of $\mu-$measurable function $f\in L^{r, s}(\Omega\cap B_R )$  such that 
\begin{equation}
\Vert f \Vert_{\mathcal{M}^{\ell}_{r, s}(\Omega,d\mu)}=\sup_{x\,\in\, \text{spt}(\mu),\, R>0}R^{-\beta \left(\frac{1}{r}-\frac{1}{\ell}\right)} \Vert f\Vert_{L^{r, s}(\Omega\cap B_R )}<\infty\label{morrey2},
\end{equation} 
where $L^{r, s}(\Omega\cap B_R )$ denotes the Lorentz space (see Section \ref{Lorentz}) defined by 
\begin{equation}
\Vert f\Vert_{L^{r, s}(\Omega\cap B_R)}=\left(r\int_0^{\infty}\left[t^rd_f(t)\right]^{\frac{s}{r}}\frac{dt}{t}\right)^{\frac{1}{s}}\nonumber,
\end{equation}
where $d_{f}(t)=\mu\lfloor_{B_R}(\{x\in \Omega\,:\,|f(x)|>t\})$ and  $\mu\lfloor_{B_R}(\Omega) =\mu (\Omega\cap B_R )$. According to \cite[Theorem 5.1]{Adams} if  the growth condition \eqref{non-doubling} holds and $\,1<p<q<\infty$ satisfies 
\begin{equation}\label{Triang}
\frac{q}{\lambda_\ast}\leq \frac{p}{\lambda},\quad  
0< \alpha<\frac{n}{\lambda},\quad  
n-\alpha p<\beta\leq n\;\text{ and } \; 
\frac{\beta}{\lambda_\ast}=\frac{n}{\lambda}-\alpha,
\end{equation} 
then 
\begin{equation}\label{weak-map}
I_{\alpha}:  \mathcal{M}_{p}^{\lambda}(\mathbb{R}^n,d\nu)\rightarrow \mathcal{M}_{q, \infty}^{\lambda_\ast}(\Omega,d\mu)
\end{equation} 
is a bounded operator. Since Morrey space is not closed by \textit{real interpolation}, the weak-trace theorem \cite[Theorem 5.1]{Adams} does not imply the strong trace version 
\begin{equation}\label{strong}
\left\Vert I_{\alpha}f\right\Vert _{\mathcal{M}^{\lambda_\ast}_{q}(d\mu)}\leq C\Vert f\Vert _{\mathcal{M}_{p}^{\lambda}(d\nu)}.
\end{equation}
However, by using the inequality Lemma \ref{lemma-point}-(i), continuity of fractional maximal function $M_{\gamma}: L^p(\mathbb{R}^n)\rightarrow L^{p\beta/(n-\gamma p)}(\Omega,d\mu)$ and atomic decomposition theorem in Hardy-Morrey space $\mathfrak{h}_p^{\lambda}(d\nu)=\mathcal{HM}_p^{\lambda}(\mathbb{R}^n,d\nu)$, 
\begin{equation*}
\Vert f\Vert_{\mathfrak{h}_p^{\lambda}}=\Big\Vert \sup_{t\in (0,\infty)} \vert \varphi_t\ast f\vert \Big\Vert_{\mathcal{M}_p^{\lambda}}<\infty\; \text{ with } \; \varphi_t=t^{-n}\varphi(x/t) \text{ for }\varphi \in \mathcal{S}(\mathbb{R}^n) \text{ and }f\in\mathcal{S}'(\mathbb{R}^n),
\end{equation*}  
Liu and Xiao \cite[Theorem 1.1]{Liu} showed that  $I_{\alpha}:  \mathfrak{h}_p^{\lambda}(d\nu)\rightarrow \mathcal{M}_{q}^{\lambda_\ast}(d\mu)$ is continuous if and only if  the Radon measure $\mu$ satisfy $\llbracket\mu\rrbracket_{\beta}<\infty$, provided that $1\leq p<q<\infty$ satisfies \eqref{Triang}. In particular, Liu and Xiao shown the strong trace inequality \eqref{strong} and since $\mathcal{M}^{\lambda_{\ast}}_{q}(d\mu)\subset \mathcal{M}^{\lambda_{\ast}}_{q, \infty}(d\mu)$ they get immediately a version of \eqref{strong} with weak-Morrey norm in the left-hand side. However, according to Sawano et al. \cite[Theorem\,1.2]{Gunawan} there is a function $g\in \mathcal{M}_{p, \infty}^{\lambda}(\mathbb{R}^n)$ such that $g\notin\mathcal{M}_{p}^{\lambda}(\mathbb{R}^n)$ and \cite[Theorem 1.1]{Liu} cannot  recover this case. This motivates us to study trace inequality in Morrey-Lorentz spaces. In particular,  under previous assumptions \eqref{Triang} we show that 
\begin{equation}
I_{\alpha}:  \mathcal{M}_{p, \infty}^{\lambda}(\mathbb{R}^n,d\nu)\rightarrow \mathcal{M}_{q, \infty}^{\lambda_\ast}(\Omega,d\mu)\nonumber
\end{equation}
is continuous if and only if  the Radon measure $\mu$ satisfy $\llbracket\mu\rrbracket_{\beta}<\infty$. Then we provide a new class of data for the trace theorem (see Theorem \ref{measure-restriction-thm}). The Lorentz space $L^{p, \infty}$ and functions space based on $L^{p, \infty}$ have been successful applied to study existence and uniqueness of \textit{mild solutions} for Navier-Stokes equations. The main effort in these works is to prove a bilinear estimate  
\begin{equation}
\Vert B(u,v)\Vert_{L^{\infty}((0,\infty); X)}\lesssim \Vert u\Vert_{L^{\infty}((0,\infty); X)}\,\Vert v\Vert_{L^{\infty}((0,\infty); X)}
\end{equation}
without invoke Kato's approach, see \cite{Lucas-bilinear} for weak-Morrey spaces, see \cite{Jhean-Ferreira} for Besov-weak-Morrey spaces  and  see \cite{yamazaki} for weak-$L^p$ spaces. For stationary Boussinesq equations, see  \cite{LJE} for Besov-weak-Morrey spaces  and  see \cite{Elder-Lucas} for weak-$L^p$ spaces. 

Choosing  a specific  $\mathfrak{h}_p^{\lambda}-$atom and using discrete Calder\'on reproducing formula in Hardy-Morrey spaces, from atomic decomposition theorem the authors \cite{Liu2} characterized the continuity of   $I_\alpha:\,\mathfrak{h}_{p}^{\lambda}(d\nu)\rightarrow \mathfrak{h}_{q}^{\lambda_\ast}(d\mu)$ by using the growth condition $\llbracket\mu\rrbracket_{\beta}<\infty$,  provided that $0<p<q<1$ satisfies \eqref{Triang}.  Meanwhile, it should be emphasized that $\mathcal{M}_{p\infty}^{\lambda}\neq \mathfrak{h}_{p}^{\lambda}$. Indeed, according to the Fourier decaying $\vert \widehat{f}(\xi)\vert \lesssim \vert \xi\vert ^{n(1/\lambda -1)}\Vert f\Vert_{\mathfrak{h}_{p}^{\lambda}(d\nu)}$ (see \cite[Corollary 3.3]{MP}) regular distributions $f\in \mathfrak{h}_{p}^{\lambda}$ satisfies $\int_{\mathbb{R}^n} f(x)dx=0$  when $0<p<\lambda<1$ which implies  $\vert x\vert^{-n/\lambda}\notin \mathfrak{h}_{p}^{\lambda}$, however $\vert x\vert^{-n/\lambda}\in \mathcal{M}_{p, \infty}^{\lambda}$.

If $d\mu$ is a doubling measure and satisfy $\llbracket\mu\rrbracket_{\beta}<\infty$, the authors of \cite[Theorem 1.1]{Liu3} showed that $I_{\alpha}$ is bounded from Besov space  $\dot{B}^s_{p,\infty}(\mathbb{R}^n,d\nu)$ to Radon-Campanato space $\mathcal{L}_{q}^{\lambda_\ast}(\mu)$ for suitable parameters $p,q,\lambda_{\ast}$ and $0<s<1$. Note that we have the continuous inclusions (see \cite[pg. 154]{BL} and \cite[Lemma 1.7]{KY})
\begin{equation}
\dot{H}_{p}^{s}\hookrightarrow\dot{B}_{p,\infty}^{s}\hookrightarrow
L^{\lambda, \infty}\hookrightarrow \mathcal{M}_{p}^{\lambda}\hookrightarrow\mathcal{M}_{p, \infty}^{\lambda}, \label{inclusion-2}%
\end{equation}
where $1<p<\lambda<\infty$ and $s\in\mathbb{R}$ satisfy $\frac{1}{p}-\frac{s}{n}=\frac{1}{\lambda}$.  In fact, the inclusions in (\ref{inclusion-2}) are strict and
then $\mathcal{M}_{p, \infty}^{\lambda}$ is strictly larger than Besov space $\dot
{B}_{p,\infty}^{s}$. So, our Theorem \ref{measure-restriction-thm} extends the previous trace results even when $d\mu$ is a non-doubling measure. 

\begin{theorem} \label{measure-restriction-thm}Let $1<p\leq \lambda<\infty$ and $1<q\leq\lambda_\ast<\infty$ be such that  $q/\lambda_\ast\leq p/\lambda$ for all $n-\delta p<\beta\leq n$ and $1<p<q<\infty$.  Then
 \begin{equation}
\Vert I_{\delta}f\Vert_{\mathcal{M}_{q, s}^{\lambda_{\ast}}(d\mu)}\lesssim \llbracket\mu\rrbracket_{\beta}^{{1}/{q}}\,\Vert f\Vert_{\mathcal{M}_{p, \ell}^{\lambda}(d\nu)}\nonumber\\[0.02in]
\end{equation}
if and only if the Radon measure $d\mu$ satisfy $\llbracket\mu\rrbracket_{\beta}<\infty$, provided that 
 $\;\,\delta=\frac{n}{\lambda}-\frac{\beta}{\lambda_\ast}$,  $\;\,0< \delta<{n}/{\lambda}\;$ and $\;1\leq \ell<s\leq \infty$.
\end{theorem}

A few remarks are in order. 
\begin{remark}\
 \begin{itemize}
		\item[(i)]\textbf{(Hardy-Littlewood-Sobolev)} Theorem \ref{measure-restriction-thm} implies 
		\begin{equation}
		\Vert I_{\delta}f\Vert_{\mathcal{M}_{q, s}^{\lambda_{\ast}}}\lesssim \llbracket\nu\rrbracket_{n}^{{1}/{q}}\,\Vert f\Vert_{\mathcal{M}_{p,\ell}^{\lambda}}\nonumber
		\end{equation}
		for $d\mu=d\nu$ and $\beta=n$. So, our theorem extend Hardy-Littlewood-Sobolev \cite[Theorem 9]{Ossen} for weak-Morrey spaces. However the optimality of  $q/\lambda_\ast\leq p/\lambda$ is known only for Morrey spaces  \cite[Theorem 10]{Ossen}. 
		\item[(ii)]\textbf{(Regularity on Morrey spaces)} If $u$ is a weak solution to fractional Laplace equation $(-\Delta)^{\frac{\delta}{2}}u=f$ in $\mathbb{R}^n$,
		\begin{equation*}
		(-\Delta)^{\frac{\delta}{2}}u(x):=C(n,\delta)\,\textbf{P.V.}\int_{\mathbb{R}^{n}}\frac{u(x)-u(y)}{\vert x-y\vert^{n+\delta}}d\nu(y)\quad \text{ with } \quad 0<\delta<2,
		\end{equation*}  
		then $u\in \mathcal{M}_{q,s}^{\lambda_{\star}}(\Omega,\,d\mu)$ if provided that $f\in \mathcal{M}_{p,\ell}^{\lambda}(\mathbb{R}^n,d\nu)$. 
		Indeed,  $u=I_{\delta}f$ is a weak solution of $(-\Delta_x)^{\frac{\delta}{2}}u=f$ because 
		\begin{equation}
		\left\langle (-\Delta)^{{\delta}/{2}}u,\, \widehat{\varphi}\right\rangle =\int_{\mathbb{R}^n}  \widehat{u}(\xi)\vert \xi\vert ^{\delta}\varphi(\xi)d\xi =\int_{\mathbb{R}^n}  \widehat{f}(\xi)\varphi(\xi)d\xi=\left\langle f, \widehat{\varphi}\right\rangle\nonumber
		\end{equation}
		for all $\varphi\in \mathcal{S}(\mathbb{R}^n)$. Then, Theorem \ref{measure-restriction-thm} give us  desired result.
		\item[(iii)]\textbf{(Adams' trace to surface-carried measures)} 
		Let $\Omega$ be a compact smooth surface with nonnegative second fundamental form and  
		$$\widehat{d\mu}(\xi)=\int_{\Omega}e^{-2\pi i x\cdot \xi} d\mu$$
		 the Fourier transform of  a measure $\mu$ supported on $\Omega$. If $\Omega$ has at  least $k$ non-vanishing principal curvatures at $\text{spt}(\mu)$, the stationary phase method (see Stein and Shakarchi \cite[Chapter 8]{Stein2.1}) gives the optimal decay 
		 $$\vert \widehat{d\mu}(\xi)\vert \lesssim \vert \xi\vert^{-\frac{k}{2}} \quad  \text{ as } \quad \vert\xi\vert>1.$$
Let $\,\phi\in \mathcal{S}(\mathbb{R}^n)$ be nonnegative, $\phi\gtrsim 1$ on $B(0,1)$ and $\widehat{\phi}=0$ on $\mathbb{R}^n\backslash B(0,R)$ for some $R>0$. 	Choosing $\phi_{x,r}(y)=\phi(\frac{x-y}{r})$  we have 
		\begin{align}
		\vert \mu(B_r(x))\vert &\lesssim \left\vert \int_{\mathbb{R}^n}\phi_{x,r}(y)d\mu(y) \right\vert = \left\vert \int_{\mathbb{R}^n}\widehat{\phi}_{x,r}(\xi)\widehat{\mu}(-\xi)d\xi \right\vert %=r^n \int_{\vert r\xi\vert\leq \delta}\widehat{\phi}(r\xi)\widehat{\mu}(-\xi)d\xi
		\nonumber\\
		  &\leq \int_{\vert \xi\vert\leq R}\vert \widehat{\phi}(\xi)\vert \vert \widehat{\mu}(-\xi/r)\vert d\xi\nonumber\\
		  &\lesssim r^{\frac{k}{2}}\int_{\vert \xi\vert\leq R}\vert \widehat{\phi}(\xi)\vert \,\vert \xi\vert^{-\frac{k}{2}}d\xi\nonumber\\
		  &\lesssim r^{k/2}\quad \text{ for all }\;\; x\in \text{spt}(\mu)\nonumber.
		\end{align}
It follows from Theorem \ref{measure-restriction-thm} that 
$\Vert I_{\delta}f\Vert_{\mathcal{M}_{q,s}^{\lambda_{\ast}}(\Omega,d\mu)}\lesssim \llbracket\mu\rrbracket_{k/2}^{{1}/{q}}\,\Vert f\Vert_{\mathcal{M}_{p,\ell}^{\lambda}}$, if provided that $f\in {\mathcal{M}_{p,\ell}^{\lambda}}$.
%\item[(iv)] Let support $\text{spt}(\mu)$ be a bounded domain $\Omega$ such that $\vert B_r(x_0)\cap \Omega\vert \geq C r^n$,  for  $x_0\in\Omega$ and $r\in (0, \text{diam}\,\Omega)$. According to  \cite[Theorem 6.II]{camp} the Campanato space $\mathcal{L}_{q}^{\lambda_{\ast}}(\Omega, \mu)$ is isomorphic to  $\mathcal{M}_q^{\lambda_{\ast}}(\Omega, \mu)$, provided $1\leq q<\lambda_{\ast}<n$.
 \end{itemize}
\end{remark}
Employing non-doubling Calderon-Zygmund decomposition \cite{Tolsa} we obtain the suitable  ``good-$\lambda$ inequality'' (see \eqref{keyM})
$$
\sum_j \mu(Q_j^t)\leq \mu(\{ x\,:\, (I_\alpha f)^{\sharp}(x)>3\epsilon t/4\})+\epsilon \sum_j \mu(Q_j^s)\; \text{ with }\;s=4^{-n-2}t\nonumber
$$
provided that $\mu$ satisfy \eqref{non-doubling}, where $(I_{\alpha}f)^{\sharp}$ denotes the (noncentered) sharp maximal function and $\{Q_j^t\}$ is a family of  doubling cubes, see Section \ref{CZ-cubes}. Then, by a suitable analysis we have the norm equivalence (see Theorem \ref{equiv-pot-frac-ML}) 
\begin{equation}\label{equiv-trace}
\Vert M_{\alpha}f\Vert_{\mathcal{M}_{p,\ell}^{\lambda}(d\mu)}\,\sim\, 	\Vert I_{\alpha}f\Vert_{\mathcal{M}_{p,\ell}^{\lambda}(d\mu)}
\end{equation}
provided that Radon measure $\mu$ such that $\mu(B_r(x))\sim r^{\beta}$  for all $x\in \text{spt}(\mu)$, where $0<\alpha<n$ satisfy $n-\alpha <\beta\leq n$  and $M_{\alpha}$ is defined as  (centered fractional maximal function)
\begin{equation}
M_{\alpha}f(x) = \sup_{r>0} r^{\alpha - n}\int_{\vert y-x\vert <r}\vert f(y)\vert d\nu \nonumber,
\end{equation}
for all locally integrable function $f\in L^1_{loc}(\mathbb{R}^n,d\nu)$. It should be emphasized that (\ref{equiv-trace}) is understood in sense of trace, since $M_\alpha$ and $I_\alpha$ are defined in $f\in L^1_{loc}(\mathbb{R}^n,d\nu)$ with Lebesgue measure $d\nu$. In particular, when $d\mu$ coincide with the Lebesgue measure $d\nu$, this equivalence recovers \cite[Theorem 4.2]{Adams2} for Morrey spaces. The proof of (\ref{equiv-trace}) is involved, because it requires a suitable analysis of non-doubling Calderon-Zygmund decomposition to yield ``good-$\lambda$ inequality" (see Lemma \ref{pot}) as well as the suitable pointwise estimate (see Lemma \ref{pointwise}) 
\begin{equation}
\overline{M}^{\sharp}I_{\alpha}f(x)\lesssim M_{\alpha}f(x)\nonumber
\end{equation}
whenever $\mu(B_r(x))\sim r^{\beta}$, where $\overline{M}^{\sharp}$ denotes the (centered) sharp maximal function. Note that (\ref{equiv-trace})  and Theorem \ref{measure-restriction-thm} yields a trace principle for $M_{\delta}$ if and only if  $\mu(B_r(x))\sim r^{\beta}$. However, the ``if part'' of trace principle for $M_{\delta}$ can be obtained directly from pointwise inequality $M_\delta f(x)\lesssim I_{\delta}\vert f(x)\vert$ and Theorem \ref{measure-restriction-thm}. The ``only if part''  is derived from the same technique used in Section \ref{only if part}. 

\begin{corollary}[Trace principle for $M_{\delta}$] \label{measure-restriction-thm2} Let $1<p\leq \lambda<\infty$ and $1<q\leq\lambda_\ast<\infty$ be such that  $q/\lambda_\ast\leq p/\lambda$ for all $n-\delta p<\beta\leq n$ and $1<p<q<\infty$.  Then,
	\begin{equation*}
	M_{\delta}:\mathcal{M}_{p,\ell}^{\lambda}(\mathbb{R}^n,d\nu)\longrightarrow\mathcal{M}_{q, s}^{\lambda_{\ast}}(\Omega,\,d\mu) \;\text{ is continuous}
	\end{equation*}
	if and only if $\,\llbracket\mu\rrbracket_{\beta}<\infty$,   for all
	$\;\,\delta=\frac{n}{\lambda}-\frac{\beta}{\lambda_\ast}$,  $\;\,0< \delta<{n}/{\lambda}\;$ and $\;1\leq \ell<s\leq \infty$. 
\end{corollary}

%It is well-known that embedding $\dot{L}^{p}_{\alpha}(d\nu) \subset L^{q}(d\mu)$ is equivalent to the characterization of measures $\mu$ for which 
%\begin{equation}\label{key_1}
%\int_{\Omega}\vert I_{\alpha}f(x)\vert ^q d\mu \lesssim \llbracket\mu\rrbracket_{\beta} \left( \int_{\mathbb{R}^n}\vert f(x)\vert^pd\nu\right)^{{q}/{p}} 
%\end{equation}
%for all $f\in L^p(d\nu)$. 

It is worth noting from an integral representation formula that $\Vert I_{k}f\Vert_{L^{q}(d\mu)} \lesssim \llbracket\mu\rrbracket_{\beta} \Vert f\Vert_{L^{p}(\mathbb{R}^n)}$ is equivalent to the trace inequality (see \cite[Corollary, p.67]{Mazya})
\begin{equation}\label{trace-lp}
\left(\int_{\Omega}\vert f(x)\vert ^qd\mu \right)^{\frac{1}{q}}\lesssim  \llbracket\mu\rrbracket_{\beta} \Vert f\Vert_{W^{k,p}(\mathbb{R}^n)} 
\end{equation}  
where $\Vert f\Vert_{W^{k,p}(\mathbb{R}^n)}=\sum_{\vert \gamma\vert \leq k}\Vert D^\gamma f\Vert_{L^{p}(\mathbb{R}^n)}$ for all $1<p<q<\infty$ and $\beta= q(n/p-k)>0$ with $0<k<n$. If $\Omega$ is a $W^{k,p}$-extension domain, that is, if there is a bounded linear operator $\mathcal{E}_k:W^{k,p}(\Omega)\rightarrow W^{k,p}(\mathbb{R}^n)$ such that $\mathcal{E}_kf|_{\Omega}=f$ for all $f\in W^{k,p}(\Omega)$, then \eqref{trace-lp} yields Sobolev trace inequality 
\begin{equation}
\left(\int_{\Omega}\vert f(x)\vert ^q d\mu\right)^{\frac{1}{q}} \lesssim   \llbracket\mu\rrbracket_{\beta} \Vert f\Vert_{W^{k,p}(\Omega)}\nonumber
\end{equation} 
provided that $\mu$ is a measure on $\Omega$ such that $\sup_{x\in\mathbb{R}^n,\,r>0}r^{-\beta}\mu(\Omega\cap B_r(x))<\infty$. From \cite[Theorem 5, p.181]{Stein} it is known that Lipschitz domain is a $W^{k,p}$-extension domain. Moreover, it is also known that $(\epsilon,\delta)$-locally uniform domain is a $W^{k,p}$-extension domain for all $1\leq p\leq \infty$ and $k\in \mathbb{N}$ (see \cite{Jones} and \cite{Roger}). Let us move to Sobolev-Morrey space $\mathcal{W}^{1,p}(\Omega)$ which is defined by 
\begin{equation}
\Vert f\Vert_{\mathcal{W}^{1,p}(\Omega)}=\sup _{x\in \Omega, r>0}\left(r^{p-n}\int_{B_r(x)\cap\Omega}\vert \nabla f\vert^pd\nu\right)^{1/p}\nonumber
\end{equation}
for all $f\in L^1_{loc}(\Omega)$ and $p\in[1,n]$. Employing a slight modification to the extension operator $\mathcal{E}_k$ of Jones \cite{Jones}, the authors of \cite{ext-Morrey} showed that an  $\epsilon$-uniform domain is a $\mathcal{W}^{1,p}$-extension domain. Since $\Omega=\mathbb{R}^n_+$ is a uniform domain, if $d\mu$ is  supported on  $\partial\mathbb{R}^n_+$ then from Theorem \ref{measure-restriction-thm} (or \cite[Theorem 1.1]{Liu}) in Morrey spaces with $\beta=n-1$ and integral representation formula \cite[(3.5)]{Adams} we obtain the Sobolev-Morrey trace inequality:
\begin{equation}\label{sobolev-trace-ineq}
	\left\Vert f(x^{\prime},0)\right\Vert _{\mathcal{M}_{q}^{\frac{\lambda(n-1)}{n-\lambda}}(\partial\mathbb{R}^{n}_+,dx')}\leq C\left\Vert\nabla f\right\Vert _{\mathcal{M}_{p}^\lambda(\mathbb{R}^{n}_{+})}
\end{equation} 
provided that $1<p\leq \lambda<n$ and $p<q\leq {\lambda(n-1)}/{(n-\lambda)}$. However we cannot apply directly \cite[Theorem 1.5(i)]{ext-Morrey} to yield \eqref{sobolev-trace-ineq}, since $\Vert f\Vert_{\mathcal{W}^{1,p}(\mathbb{R}^{n}_{+})}= \Vert \nabla f\Vert_{\mathcal{M}_p^{n}(\mathbb{R}^{n}_{+})}$ and $\lambda=n$. One could ask: does the Sobolev trace embedding \eqref{sobolev-trace-ineq}  holds for Morrey spaces or weak-Morrey spaces? As a byproduct of Theorem \ref{measure-restriction-thm}  and Calder\'on-Stein's extension on half-spaces (see Lemma \ref{lem-ext1}) we give a positive answer for this question.
\begin{corollary}[Sobolev-Morrey trace]\label{Trace} Let $\,1<p\leq \lambda<n$ and $1<q\leq\lambda_\ast<\infty$  be such that  $\,\frac{n-1}{\lambda_{\ast}}=\frac{n}{\lambda}-1\,$ and  $\,q/\lambda_\ast\leq p/\lambda$. Then 
	\begin{equation}
	\left\Vert f(x^{\prime},0)\right\Vert _{\mathcal{M}_{q, s}^{\lambda_{\ast}}(\partial\mathbb{R}^{n}_+,dx')}\leq C\left\Vert\nabla f\right\Vert _{\mathcal{M}_{p, d}^\lambda(\mathbb{R}^{n}_{+})},\nonumber
	\end{equation}
for all $1<p<q<\infty$ and $1\leq d<s\leq\infty$.
\end{corollary}

\bigskip
The paper is organized as follows. In Section \ref{Lorentz} we summarize properties of Lorentz  spaces. In Section \ref{CZ-dec} we deal with non-doubling CZ-decomposition for polynomial growth measures and estimates for sharp maximal function. In Sections \ref{thm1} and \ref{extension} we prove our main theorems.

\section{The Lorentz spaces}\label{Lorentz}
Let $(\Omega,\mathcal{B}, \mu)$ be a measure space endowed by Borel regular measure $d\mu$. The Lorentz space $L^{p, d}(\Omega,\mu)$ is defined as the set of  $\mu$-\text{measurable functions} $f:\Omega\rightarrow\mathbb{R}$ such that 
\begin{equation}\label{almost_norm1} 
\Vert f\Vert_{L^{p, d}}^{\ast}=\left(\frac{d}{p}\int_0^{\mu (\Omega)}\left[t^{{1}/{p}}f^\ast(t)\right]^d \frac{dt}{t}\right)^{\frac{1}{d}}=\left(p\int_0^{\infty}\left[s^pd_f(s)\right]^{\frac{d}{p}}\frac{ds}{s}\right)^{\frac{1}{d}}<\infty
\end{equation}
for all $1\leq p<\infty$ and $1\leq d<\infty$, where 
\[
\text{ }f^{\ast}(t)=\inf\{s>0:d_{f}(s)\leq t\} \text{ and } d_{f}(s)=\mu(\{x\in\Omega:|f(x)|>s\}). 
\]
For $ 1\leq p\leq  \infty$ and $d=\infty$ the Lorentz space $L^{p, \infty}(\Omega,\mu)$ is defined by 
\begin{align}\label{almost_norm}
\Vert f\Vert_{L^{p,\infty}}^{\ast}=\sup_{0<t<\mu (\Omega)} t^{1/p}f^\ast(t)= \sup_{0<s<\mu (\Omega)} [s^{p}d_f(s)]^{1/p}.
\end{align}
The Lorentz space $L^{p,d}(\Omega,d\mu)$ increases with the index $d$,  that is, 
\begin{equation}
L^{p,1}\hookrightarrow L^{p,d_{1}}\hookrightarrow L^{p}\hookrightarrow L^{p,d_{2}}\hookrightarrow
L^{p,\infty}\nonumber
\end{equation}
provided that $1\leq d_{1}\leq p\leq d_{2}<\infty$. More precisely, we have the following lemma.
\begin{lemma}[Calder\'on]\label{inclusion-1}If $1\leq p<\infty$ and $0<q<r\leq\infty$, then $\Vert f\Vert_{L^{p,r}}\leq (q/p)^{\frac{1}{q}-\frac{1}{r}}\Vert f\Vert_{L^{p,q}}$. 
\end{lemma}

The quantities \eqref{almost_norm1} and \eqref{almost_norm} are not a norm, however  
\begin{equation}\label{norm1} 
\Vert f\Vert_{L^{p,d}}^{{\color{red}{\natural}}}=\left(\frac{d}{p}\int_0^{\mu (\Omega)}[t^{{1}/{p}}f^{\color{red}{\natural}}(t)]^d \frac{dt}{t}\right)^{\frac{1}{d}}<\infty \text{ with } 
f^{\color{red}{\natural}}(t)=\frac{1}{t} \int_0^t f^\ast(s)ds\nonumber
\end{equation}
defines a norm in $L^{p,d}(\Omega,d\mu)$ and one has 
\begin{equation}
\Vert f\Vert_{L^{p,d}}^{\ast}\leq\Vert f\Vert_{L^{p,d}}^{{\color{red}{\natural}}}\leq\frac{p}
{p-1}\Vert f\Vert_{L^{p,d}}^{\ast}\label{equiv-Lpk}\nonumber
\end{equation}
for all  $1< p<\infty$ and $1\leq d\leq \infty$. The following lemma is well-known in theory of Lorentz spaces.
\begin{lemma}[Hunt's Theorem \cite{Sharpley}]\label{Hunt-thm} Let $(M_1,\mu_1)$ and $(M_2,\mu_2)$ be measure spaces and  let $T$ be a sublinear operator such that 
	\begin{align}
	\Vert Tf\Vert_{L^{q_i, s_i}(M_1,d\mu_1)} \leq C_i\Vert f\Vert_{L^{p_i, r_i}(M_2,d\mu_2)}\nonumber\;\; \text{ for }\;\; i=0,1
	\end{align}
for all $p_0\neq p_1$ and $q_0\neq q_1$. Let $0<\theta<1$ be such that $1/p=(1-\theta)/p_0+\theta/p_1$ and $1/q=(1-\theta)/q_0+\theta/q_1$, then
\begin{align*}
	\Vert Tf\Vert_{L^{q, s}(M_1,d\mu_1)} \leq C_0^{\theta}C_1^{1-\theta}\Vert f\Vert_{ L^{p, r}(M_2,d\mu_2)}\nonumber,
\end{align*}
provided that $p\leq q$ and $0<r\leq s\leq \infty$, where $C_i>0$ depends only on $p_i,q_i,p,q$.
\end{lemma}

\section{Maximal functions and non-doubling measure}\label{CZ-dec}
In this section we are interested in proving the estimate
\begin{equation}\label{equiv-pot-frac-ML0}
\Vert I_{\alpha}f\Vert_{\mathcal{M}_{p, \ell}^{\lambda}(d\mu)}\lesssim \Vert M^{\sharp}I_{\alpha}f\Vert_{\mathcal{M}_{p, \ell}^{\lambda}(d\mu)}\nonumber
\end{equation} 
for every Radon measure $\mu$ satisfying  \eqref{non-doubling}, where $M^{\sharp}f(x):=f^{\sharp}(x)$ denotes the uncentered sharp maximal function 
\begin{equation}\label{unc-sharp}
f^{\sharp}(x)=\sup_{Q,\; x\in Q}\left\{\frac{1}{\mu(Q)}\int_{Q}|f-f_Q|d\mu\right\}
\end{equation}
and $f_{Q}=\frac{1}{\mu(Q)}\int_{Q}fd\mu$.  Also, we are interested in proving the estimate
\begin{equation}
\overline{M}^{\sharp}I_{\alpha}f(x)\lesssim  M_{\alpha}f(x)\nonumber
\end{equation} 
when the Radon measure $\mu$ satisfy $\mu(B_r(x))\sim r^{\beta}$, where $\overline{M}^{\sharp}f(x):=f^{\overline{\sharp}}(x)$ denotes the centered sharp maximal function 
\begin{equation}\label{c-sharp}
f^{\overline{\sharp}}(x)=\sup_{r>0}\left\{\frac{1}{\mu(B_r(x))}\int_{B_r(x)}|f-f_{B_r}|d\mu\right\}.
\end{equation}

%Also, we denote by  $M_{\alpha}$ the centered fractional maximal function
%\begin{equation}M_{\alpha}f(x) = \sup_{r>0} r^{\alpha - n}\int_{\vert y-x\vert <r}\vert f(y)\vert d\nu \nonumber\end{equation}
% for all $f\in L^1_{loc}(d\nu)$.

\subsection{Non-doubling CZ-decomposition}\label{CZ-cubes}
Let us recall that a cube $Q\subset \mathbb{R}^n$ is called a $(\tau, \gamma)$-doubling cube with respect to polynomial growth measure $\mu$, if $\mu(\tau Q)\leq \gamma\,\mu(Q)$ as $\tau>1$ and $\gamma>\tau^{\beta}$. According to \cite[Remark 2.1 and Remark 2.2]{Tolsa} there are small/big  $(\tau,\gamma)-$doubling cubes in $\mathbb{R}^n$.
\begin{lemma}[\cite{Tolsa}]\label{tolsa-small/big} Let $\mu$ be a Radon measure in $\mathbb{R}^n$ with  growth condition \eqref{non-doubling}, then
	\begin{itemize}
		\item[(i)]\textbf{(Small doubling cubes)} Assume $\gamma>\tau^{n}$, then for $\mu$-a.e. $x\in\mathbb{R}^n$ there exists a sequence $\{Q_j\}_j$ of $(\tau, \gamma)$-doubling cubes centered at $x$ such that $\ell(Q_j)\rightarrow 0$ as $j\rightarrow\infty$.
		\item[(ii)]\textbf{(Big doubling cubes)} Assume $\gamma>\tau^{\beta}$, then for any $x\in\text{spt}(\mu)$ and $c>0$, there exists a  $(\tau,\gamma)$-doubling cube $Q$ centered at $x$ such that $\ell(Q)>c$.
	\end{itemize}
\end{lemma}
Let  $f\in L^1_{loc}(\mu)$ and $\lambda>\frac{1}{\mu(Q_0)}\Vert f\Vert_{L^1(Q_0)}$ be such that $\Omega_{\lambda}=\{x\in Q_0\,:\, \vert f(x)\vert>\lambda\}\neq \varnothing $. From Lemma \ref{tolsa-small/big}-(i) and Lebesgue differentiation theorem, there is a sequence of $(2,2^{n+1})-$doubling cubes $\{Q_j(x)\}_j$ with $\ell(Q_j)\rightarrow 0$ such that 
\begin{equation}
\frac{1}{\mu(Q_j)}\int_{Q_j}\vert f\vert d\mu >\lambda \nonumber
\end{equation}
for $j$ sufficiently large. Since there are big $(2,2^{n+1})-$doubling cubes $Q_j$, then  $\frac{1}{\mu(Q_j)}\int_{Q_j}\vert f\vert d\mu \leq \frac{\Vert f\Vert_{L^1(\mu)}}{\mu(Q_j)}\leq \lambda$ for $\mu(Q_j)>c$ sufficiently large. In other words, for $\mu$-almost all $x\in \mathbb{R}^n$ such that $\vert f(x)\vert>\lambda$ there is a $(2,2^{n+1})-$doubling cube $Q'\in \{Q_x\}_{x\in \Omega_{\lambda}}$ with center $x=x_Q$ such that 
\begin{equation}
\frac{1}{\mu(2Q')}\int_{Q'}\vert f\vert d\mu \leq \lambda/2^{n+1}\nonumber.
\end{equation}
Moreover, if $Q=Q(x)$ is a $(2,2^{n+1})-$doubling cube with sidelength $\ell(Q)<\ell(Q')/2$ then 
\begin{equation}
\frac{1}{\mu(Q)}\int_{Q}\vert f\vert d\mu >\lambda.\nonumber
\end{equation}
Hence, a non-doubling Calder\'on-Zygmund decomposition can be obtained. To simply, a doubling cube $Q$ mean $(2,2^{n+1})$-doubling cube.

\begin{lemma}[\cite{Tolsa} non-doubling CZ-decomposition] \label{CZ}Let the Radon measure $\mu$ satisfy \eqref{non-doubling}. Let $Q$ be a doubling cube  big so that $\lambda >\frac{1}{\mu(Q)}\int_{Q} \vert f\vert d\mu$ for $f\in L^1(\mu)(Q)$. Then, there is a sequence of  doubling cubes $\{Q_j\}_j$ such that %every point $z\in Q$ belongs at most $\theta_n$ cubes $Q_j$, i.e. $\sum_j \chi _{Q_j}(z)\leq \theta_n$ and moreover, 
	\begin{itemize}
		\item[(i)] $\vert f(x)\vert \leq \lambda\;$ for  $\;x\in Q\backslash 	\bigcup_jQ_j$,  $\;\mu$-a.e.
		\item[(ii)] $\lambda < \frac{1}{\mu(Q_j)}\int_{Q_j} \vert f\vert d\mu \leq 4^{n+1}\lambda $
		\item[(iii)] $	\bigcup_jQ_j = \bigcup_{k=1}^{\varepsilon_n}\bigcup_{Q_j^k\in \mathcal{F}_k}Q_j^k$,
	\end{itemize}
where the family $\mathcal{F}_k=\{Q^k_j\}$ is pairwise disjoint.
\end{lemma}
\begin{proof} This lemma is a consequence of Besicovitch's covering theorem and has been proved by Tolsa \cite[Lemma 2.4]{Tolsa}. Note that \cite[Lemma 2.4]{Tolsa} with $\eta=4$ implies 
	\begin{align*}
 \frac{1}{\mu(Q_j)}\int_{Q_j} \vert f\vert d\mu \leq \frac{\mu(\eta Q_j)}{\mu(Q_j)}\left( \frac{1}{\mu(\eta Q_j)}\int_{\eta Q_j} \vert f\vert d\mu\right)\leq 
 \frac{\mu(\eta Q_j)}{\mu(Q_j)}\left( \frac{2^{n+1}}{\mu(2\eta Q_j)}\int_{\eta Q_j} \vert f\vert d\mu\right)\leq 4^{n+1}\lambda,
	\end{align*}
thanks to $\mu (2\eta Q_j)\leq 2^{n+1}\mu (\eta Q_j)$ and $\mu(4Q_j)\leq 4^{n+1}\mu (Q_j)$.
\end{proof}

\subsection{Estimates for sharp maximal function}
Inspired in \cite[p.153]{Stein0} we prove the following lemma.
\begin{lemma}\label{pot} Let $\mu$ be a Radon measure in $\mathbb{R}^n$ such that  $ \llbracket\mu\rrbracket_{\beta}<\infty\,$ for $\,0<\beta\leq n$. If $\,I_{\alpha}f\in L^1_{loc}(d\mu)$ for $0<\alpha<n$, then 
	\begin{equation}
\Vert I_{\alpha}f\Vert_{\mathcal{M}_{p, \ell}^{\lambda}(d\mu)}\lesssim \Vert (  I_{\alpha}f)^{\sharp}\Vert_{\mathcal{M}_{p, \ell}^{\lambda}(d\mu)},
	\end{equation} 
	for every $1\leq p\leq\lambda<\infty$ and $1\leq \ell\leq \infty$. 
\end{lemma}
\begin{proof} Let $Q_0\subseteq\mathbb{R}^n$ be a doubling cube. Applying Lemma \ref{CZ} with $I_\alpha f\in L^1_{loc}(\mu)(Q_0)$ and $t=\lambda$ we obtain a family of almost disjoint doubling cubes $\{Q_j^{t}\}$ so that 
	\begin{equation}\label{key_t}
t < \frac{1}{\mu(Q^t_j)}\int_{Q^t_j} \vert I_{\alpha}f\vert d\mu \leq 4^{n+1}t
	\end{equation}
	and $I_{\alpha}f(x)\leq t$ as  $x\notin \bigcup_jQ^t_j$ $\;\mu$-a.e. The main inequality to be proved reads as follows
	\begin{equation}\label{keyM}
\sum_j \mu(Q_j^t)\leq \mu(\{ x\,:\, (I_\alpha f)^{\sharp}(x)>3\epsilon t/4\})+\epsilon \sum_j \mu(Q_j^s)\; \text{ with }\;{s=4^{-n-2}t}  
	\end{equation}
for all $\epsilon>0$. Indeed, let $s=4^{-n-2}t$ and $\mathcal{F}_1$ be the family of doubling cubes $\{Q_j^s\}$ of the $CZ$-decomposition associated to $s$ and satisfying 
	\begin{equation}\label{F_1}
	Q_j^s \subset \left\{ x\in Q_0\,:\, (I_\alpha f)^{\sharp}(x)>\frac{3\epsilon t}{4}\right\}
	\end{equation}
	and let $\mathcal{F}_2$ be the family of doubling cubes such that $Q_j^s \nsubseteq \{ x\in Q_0\,:\, (I_\alpha f)^{\sharp}(x)>3\epsilon t/4\}$. 
If $Q\in\mathcal{F}_2$, obviously one has  $(I_\alpha f)^{\sharp}(x) \leq \frac{3\epsilon t}{4}$ for $x\in Q$ and right-hand side of (\ref{key_t}) implies $(I_\alpha f)_Q=\frac{1}{\mu(Q)}\int_Q \vert I_\alpha f\vert d\mu \leq 4^{n+1}s=t/4$. Now from left-hand side of (\ref{key_t}) one has
\begin{align*}
\sum_{Q_j^t\subset Q}t\mu (Q_j^t)& < \sum_{Q_j^t\subset Q}\, \int_{Q_j^t}\vert I_\alpha f(x)\vert d\mu \\
&\leq  \sum_{Q_j^t\subset Q} \, \int_{Q_j^t}\vert I_\alpha f(x) - (I_\alpha f)_Q\vert d\mu +(I_\alpha f)_Q \sum_{Q_j^t\subset Q}\mu (Q_j^t)\\
&\leq  \int_{Q}\vert I_\alpha f(x) - (I_\alpha f)_Q\vert d\mu +(I_\alpha f)_Q \sum_{Q_j^t\subset Q}\mu (Q_j^t)\\
&\leq \frac{3\epsilon}{4} t\,\mu(Q) + \frac{t}{4}  \sum_{Q_j^t\subset Q}\mu (Q_j^t).
\end{align*}
Hence, summing over all cubes $Q\in \mathcal{F}_2$, we have
\begin{equation}
\sum_{Q\in\mathcal{F}_2}\sum_{Q_j^t\subset Q}\mu (Q_j^t)\leq \epsilon \sum_{Q\in\mathcal{F}_2}\mu (Q).\label{keyF2}
\end{equation}
If $Q\in\mathcal{F}_1$, trivially (\ref{F_1}) gives us
\begin{align}
\sum_{Q\in\mathcal{F}_1}\sum_{Q_j^t\subset Q}\mu (Q_j^t)&\leq \sum_{Q\in\mathcal{F}_1} \mu \left(\left\{ x\in Q_0\,:\, (I_\alpha f)^{\sharp}(x)>3\epsilon t/4\right\}\cap Q\right)\nonumber\\
&\leq \mu \left(\left\{ x\in Q_0\,:\, (I_\alpha f)^{\sharp}(x)>3\epsilon t/4\right\}\right) \label{keyF1}.
\end{align}
Since 
\begin{equation}
	\sum_j \mu(Q_j^t)= \left(\sum_{Q\in\mathcal{F}_1} +\sum_{Q\in\mathcal{F}_2}\right)\sum_{Q_j^t\subset Q}\mu(Q_j^t),\nonumber
\end{equation}
from estimates (\ref{keyF2}) and (\ref{keyF1}) we obtain the good-$\lambda$ inequality \eqref{keyM}. 

Now let $d_{I_\alpha f }(t)=\mu(\{x\in Q_0\,:\, I_{\alpha}f(x)>t\})$ be the distribution function of $I_{\alpha}f$, then by CZ-decomposition we have 
\begin{equation}
d_{I_\alpha f }(t)\leq\rho(t){:=\sum_j \mu(Q_j^t)}\nonumber
\end{equation}
thanks to Lemma \ref{CZ}-(i). Now invoke \eqref{keyM} in order to infer
\begin{align*}
p\int_0^{\infty} t^{\ell-1}\left[\rho(t)\right]^{\frac{\ell}{p}}dt&\lesssim p\int_0^{\infty} t^{\ell-1}\left[d_{(I_{\alpha}f)^{\sharp}}(3\epsilon t/4)\right]^{\frac{\ell}{p}}dt+p\int_0^{\infty} t^{\ell-1}[\epsilon \rho(4^{-n-2} t)]^{\frac{\ell}{p}}dt\\[0.1in]
&=(4/3\epsilon)^{\ell}p\int_0^{\infty}t^{\ell-1} \left[d_{(I_{\alpha}f)^{\sharp}}(t)\right]^{\frac{\ell}{p}}dt+4^{(n+2)\ell}\epsilon^{\frac{\ell}{p}}p\int_0^{\infty}t^{\ell-1}[\rho(t)]^{\frac{\ell}{p}}dt.
\end{align*}
Now choosing $\epsilon>0$ in such a way that $\epsilon^{\frac{\ell}{p}}4^{(n+2)\ell}=1/2$  we obtain 
\begin{equation}
\frac{p}{2}\int_0^{\infty} t^{\ell-1}\left[\rho(t)\right]^{\frac{\ell}{p}}dt\lesssim  p\int_0^{\infty}t^{\ell-1} \left[d_{(I_{\alpha}f)^{\sharp}}(t)\right]^{\frac{\ell}{p}}dt\nonumber.
\end{equation}
Since $d_{I_\alpha f }(t)\leq\rho(t)$ we estimate 
\begin{align*}
	\Vert I_{\alpha}f\Vert_{\mathcal{M}_{p, \ell}^{\lambda}(d\mu)}^\ell&=\sup_{x\,\in \text{spt}(\mu), R>0}R^{-\ell\beta \left(\frac{1}{p}-\frac{1}{\lambda}\right)} \left(p\int_0^{\infty} t^{\ell-1}\left[d_{I_{\alpha}f}(t)\right]^{\frac{\ell}{p}}dt\right)\\
	&\lesssim \sup_{x\,\in \text{spt}(\mu), R>0}R^{-\ell\beta \left(\frac{1}{p}-\frac{1}{\lambda}\right)} \left( p\int_0^{\infty}t^{\ell-1} \left[d_{(I_{\alpha}f)^{\sharp}}(t)\right]^{\frac{\ell}{p}}dt\right)\\
	&=	\big\Vert (I_{\alpha}f)^{\sharp}\big\Vert_{\mathcal{M}_{p,\ell}^{\lambda}(d\mu)}^\ell,
\end{align*}
as required. The case $\ell=\infty$ is achieved without great effort.
\end{proof}

\begin{lemma} \label{pointwise}Let $\mu$ be a Radon measure  such that $\mu(B_r(x))\sim r^{\beta}\,$
for all $x\in\mathbb{R}^n$ and $r>0$. If  $f\in L^1_{loc}(d\nu)$ is such that $I_{\alpha}f\in L^{1}_{loc}(d\mu)$ when $0<\alpha<n$ satisfy $n-\alpha <\beta\leq n$, then 
\begin{equation}
\overline{M}^{\sharp}I_{\alpha}f(x)\lesssim  M_{\alpha}f(x)\nonumber.
\end{equation}
\end{lemma}
\begin{proof}Taking $f'=f\chi_{B(x_0,2r)}$ and $f''=\chi_{\mathbb{R}^n\backslash B(x_0,2r)}$ from Fubini's theorem and   \cite[Lemma 3.1.1]{Adams-Hedberg} we estimate 
	\begin{align*}
	\int_{\vert x-x_0\vert <r}\vert I_{\alpha}f'(x)\vert d\mu(x)
	&\lesssim 	\int_{\vert x-x_0\vert <r}\left(\int_{\vert y-x_0\vert <2r}\vert y- x\vert^{\alpha-n}\vert f(y)\vert d\nu \right)d\mu(x)\\
	&\leq \int_{\vert y-x_0\vert <2r}\left(\int_{\vert y-x\vert <3r}\vert y- x\vert^{\alpha-n}d\mu(x) \right)\vert f(y)\vert d\nu\\
	&\leq \int_{\vert y-x_0\vert <2r}\left[ (n-\alpha)\int_0^{3r}\frac{\mu(B(x,s))}{s^{n-\alpha}}\frac{ds}{s}+\frac{\mu(B(x,3r))}{(3r)^{n-\alpha}}\right]\vert f(y)\vert d\nu\\
	&\lesssim \llbracket\mu\rrbracket_{\beta} r^{\beta} [2r]^{\alpha-n}\int_{\vert y-x_0\vert <2r}\vert f(y)\vert d\nu\\
%	&\lesssim \llbracket\mu\rrbracket_{\beta}\,r^{\beta}\, [2r]^{\alpha -n}\int_{\vert y-x_0\vert<2r}\vert f(y)\vert d\nu\\
%	&\leq  \mu(B(x,3r)) M_{\alpha}f(x_0)\\
		&\lesssim  \llbracket\mu\rrbracket_{\beta}\,\mu(B_r(x_0))\, M_{\alpha}f(x_0),
	\end{align*}
which yields $ \overline{M}^{\sharp}I_{\alpha}f'(x_0)\lesssim   \llbracket\mu\rrbracket_{\beta} M_{\alpha}f(x_0). $
Now from mean value theorem we have 
$$\left \vert \vert x-z\vert^{\alpha-n}-\vert y-z\vert^{\alpha-n}\right \vert\lesssim r\, \vert z-x_0\vert^{\alpha -n-1},$$ 
for $\vert x-x_0\vert<r$ and $\vert y-x_0\vert<r$. Hence, Fubini's theorem implies 
\begin{align}
\left\vert (I_{\alpha}f'')(x) - (I_\alpha f'')_{B_r(x_0)}\right \vert 
%&\leq \frac{1}{\mu (B_r)}\int_{\vert z -x_0\vert >2r}\left\{ \int_{\vert y-x_0\vert <r}\left( \vert x-z\vert^{\alpha-n}-\vert y-z\vert^{\alpha-n}\right)d\mu(y)\right\}\vert f(z)\vert dz\nonumber \\
&\leq \frac{1}{\mu (B_r)}\int_{\vert z -x_0\vert >2r}\left\{ r\int_{\vert y-x_0\vert <r}\vert z-x_0\vert^{\alpha-n-1} d\mu(y)\right\}\vert f(z)\vert d\nu\nonumber \\
&\lesssim r\,\int_{\vert z -x_0\vert >2r}\vert z-x_0\vert^{\alpha-n-1}\vert f(z)\vert d\nu\nonumber \\
&= r\sum_{k=1}^{\infty} \int_{2^{k}r\leq \vert z-x_0\vert < 2^{k+1}r}\vert z-x_0\vert^{\alpha-n-1}\vert f(z)\vert d\nu\nonumber\\
%&\leq \sum_{k=1}^{\infty}2^{-(k+1)} \left(2^{k+1}r\right)^{\alpha-n}\int_{\vert z-x_0\vert < 2^{k+1}r}\vert z-x_0\vert^{\alpha-n-1}\vert f(z)\vert d\nu\nonumber\\
&\leq \sum_{k=1}^{\infty}2^{-(k+1)} M_{\alpha}f(x_0)\lesssim M_{\alpha}f(x_0), \nonumber
\end{align}
which yields %Since $\mu (B_r(x_0))\lesssim r^{\beta}$, then 
\begin{align}
\overline{M}^{\sharp}I_{\alpha}f''(x_0)=\sup_{r>0}\,\frac{1}{\mu(B_r(x_0))}\int_{\vert x-x_0\vert <r} \left\vert (I_{\alpha}f'')(x) - (I_\alpha f'')_{B(x_0,r)}\right \vert d\mu(x)\lesssim  %\llbracket\mu\rrbracket_{\beta} \,
M_{\alpha}f(x_0)\nonumber,
\end{align}
as required.
\end{proof}

\begin{theorem}[Trace-type equivalence]\label{equiv-pot-frac-ML} Let $\mu$ be a Radon measure  such that $\mu(B_r(x))\sim r^{\beta}\,$
for all $x\in\mathbb{R}^n$ and $r>0$. If  $f\in L^1_{loc}(d\nu)$ is such that $I_{\alpha}f\in L^{1}_{loc}(d\mu)$ whenever  $0<\alpha<n$ satisfy  $n-\alpha <\beta\leq n$, then
	\begin{equation}
\Vert M_{\alpha}f\Vert_{\mathcal{M}_{p,\ell}^{\lambda}(d\mu)}\,\sim\, 	\Vert I_{\alpha}f\Vert_{\mathcal{M}_{p,\ell}^{\lambda}(d\mu)}\nonumber,
	\end{equation}
for all $1\leq p\leq\lambda<\infty$ and $1\leq \ell\leq \infty$. 
\end{theorem}
\begin{proof}Note that Lemma \ref{pot} is true with centered sharp maximal function $\overline{M}^{\sharp}I_{\alpha}f$. Since $M_\alpha f(x)\lesssim I_{\alpha}f(x)$, by  Lemma \ref{pot} and   \ref{pointwise} we obtain
\begin{equation}
\Vert M_\alpha f\Vert_{\mathcal{M}_{p,\ell}^{\lambda}}\lesssim \Vert I_\alpha f\Vert_{\mathcal{M}_{p,\ell}^{\lambda}} \stackrel{\text{Lemma } \ref{pot}}{\lesssim} \left\Vert\overline{M}^{\sharp}I_{\alpha}f\right\Vert_{\mathcal{M}_{p,\ell}^{\lambda}}\stackrel{\text{Lemma } \ref{pointwise}}{\lesssim} \Vert M_\alpha f\Vert_{\mathcal{M}_{p,\ell}^{\lambda}},
\end{equation}
which is the desired result.
\end{proof}

\section{Proof of trace Theorem \ref{measure-restriction-thm}}\label{thm1}
Let us recall the pointwise estimate between Riesz potential and fractional maximal operator.

\begin{lemma}\label{lemma-point} Let $f\in L^1_{loc}(\mathbb{R}^n,d\nu)$ and $B(x,r)\subset \mathbb{R}^n$ a ball with radius $r>0$.
	
\begin{itemize}
	\item[(i)] If $0\leq \gamma<\delta<\alpha\leq n$, then 
		\begin{equation}
	\left\vert I_{\delta}f(x)\right\vert \lesssim \left[M_{\alpha}f(x)\right]^{\frac{\delta-\gamma}{\alpha-\gamma}}   \left[M_{\gamma}f(x)\right]^{1-\frac{\delta-\gamma}{\alpha-\gamma}},\;\; \forall\, x\in\mathbb{R}^n \nonumber.
	\end{equation}
	\item[(ii)] If $1\leq p<\infty$ and $1\leq k\leq\infty$, then 
	\begin{equation}
	[{\nu(B(x,r))}]^{\frac{1}{p}-1}\int_{B(x,r)}\vert f(y)\vert d\nu \lesssim \Vert f\Vert_{L^{p,k}(B(x,r))}.\nonumber
	\end{equation}	
\end{itemize}	
	In particular,  $(M_{n/\lambda}f)(x)\lesssim \Vert f\Vert_{\mathcal{M}_{p,k}^{\lambda}(d\nu)}$, for all $p\leq\lambda<\infty$.
\end{lemma}
\begin{proof} The item $(i)$ was obtained in \cite[Lemma 4.1]{Liu}. To show $(ii)$, first let us recall the Hardy-Littlewood inequality
	\begin{align*}
	\int_{B(x,r)}\vert f(y)g(y)\vert d \nu\leq \int_0^{\nu(B(x,r))}f^{\ast}(t)g^{\ast}(t)dt.
	\end{align*}
This inequality and H\"older's inequality in $L^{k}(\mathbb{R},dt/t)$ give us
\begin{align*}
\int_{B(x,r)}\vert f(x)\vert d\nu &\leq \int_0^{\nu(B(x,r))}t^{1-\frac{1}{p}}\left(t^{\frac{1}{p}}f^{\ast}(t)\right)\frac{dt}{t}\\
&\leq \left(\int_0^{\nu(B(x,r))}\left(t^{1-\frac{1}{p}}\right)^{k'}\frac{dt}{t}\right)^\frac{1}{k'}\left(\int_0^{\nu(B(x,r))}\left(t^{\frac{1}{p}}f^{\ast}(t)\right)^{k}\frac{dt}{t}\right)^\frac{1}{k}\nonumber\\
&\lesssim [\nu(B(x,r))]^{1-\frac{1}{p}}\Vert f\Vert_{L^{p,k}(B(x,r))},
\end{align*}
as desired.
\end{proof}
Now, we are in position to prove Theorem \ref{measure-restriction-thm}. 
\subsection{The condition $\llbracket\mu\rrbracket_{\beta}<\infty$ is sufficient}
 For $x\in B_\rho=B(x_0,\rho)$ with $\rho>0$, let us write  
	\begin{align*}
I_{\delta}f(x)&= \int_{|y-x|<\rho}\vert x-y\vert^{\delta-n} f(y)d\nu(y)+\int_{|y-x|\geq \rho}\vert x-y\vert^{\delta-n} f(y)d\nu(y):=I_{\delta}f'(x)+I_{\delta}f''(x), \nonumber
\end{align*}
where $f'=\chi_{B(x_0,2\rho)}f$ and $f{''}=f-f'$.  If $y\in \mathbb{R}^n\backslash B(x_0,2\rho)$, using integration by parts and Lemma \ref{lemma-point}-(ii), respectively, we have 
\begin{align}
\left\vert I_{\delta}f{''}(x)\right\vert&\leq \int_{2\rho}^{\infty} s^{\delta-n}\left(\int_{B(x,s)}\vert f(y)\vert d\nu \right)\frac{ds}{s}\nonumber\\[0.02in]
&\lesssim \int_{\rho}^{\infty} s^{\delta-n}[\nu(B(x,s))]^{1-\frac{1}{p}}\Vert f\Vert_{L^{p,\ell}(B(x,s))}\frac{ds}{s}\nonumber\\[0.02in]
&\leq \left(\int_{\rho}^{\infty} s^{\delta-1-\frac{n}{\lambda}}ds\right) \Vert f\Vert_{\mathcal{M}_{p,\ell}^{\lambda}(d\nu)}\nonumber\\[0.02in]
&\lesssim \rho^{\delta-\frac{n}{\lambda}}\Vert f\Vert_{\mathcal{M}_{p,\ell}^{\lambda}(d\nu)},\nonumber
\end{align}
in view of $0<\delta<{n}/{\lambda}$. Therefore, $(I_{\delta}f{''})^{\ast}(t)\lesssim   \rho^{\delta-\frac{n}{\lambda}}\Vert f\Vert_{\mathcal{M}_{p,\ell}^{\lambda}(d\nu)} $ and we can estimate
\begin{align}
\Vert I_{\delta}f{''}\Vert_{L^{q, s}(B(x_0,\rho),d\mu)}&\lesssim \rho^{\delta-\frac{n}{\lambda}}\Vert f\Vert_{\mathcal{M}_{p,\ell}^{\lambda}(\mathbb{R}^n,d\nu)}\left(\int_0^{\mu(B(x_0,\rho))}t^{\frac{s}{q}-1}dt\right)^{\frac{1}{s}}\nonumber\\[0.01in]
&\leq \rho^{\delta-\frac{n}{\lambda}} \mu(B(x_0,\rho))^{\frac{1}{q}}\Vert f\Vert_{\mathcal{M}_{p,\ell}^{\lambda}(d\nu)}\nonumber\\[0.01in]
&\leq \rho^{\frac{\beta}{q}-\frac{\beta}{\lambda_\ast}} \,\llbracket\mu\rrbracket_{\beta}^{\frac{1}{q}}\,\Vert f\Vert_{\mathcal{M}_{p,\ell}^{\lambda}(d\nu)},\label{fora-da-bola}
\end{align}
thanks to $\delta=\frac{n}{\lambda}-\frac{\beta}{\lambda_\ast}$ and $\mu(B(x_0,r))\leq \llbracket\mu\rrbracket_{\beta} \,r^{\beta}$, for all $\,x_0\in\text{spt}(\mu)\,$ and $\,r>0$. Since 
\begin{equation}
	\frac{n-\beta}{p}<\delta <\frac{n}{\lambda}\nonumber
\end{equation}
we can ensure the existence of  $\gamma$ such that ${(n-\beta)}/{p}<\gamma<\delta <{n}/{\lambda}$ which yields $n-\beta<\gamma p <np/\lambda \leq n$. Hence, for $y\in B(x_0,2\rho)$ we invoke Lemma \ref{lemma-point} with $\alpha=n/\lambda$ and estimate
\begin{align}
	\Vert I_{\delta}f{'}\Vert_{L^{q, s}(B(x_0,\rho),d\mu)}%&\lesssim \left\Vert \left[M_{\alpha}f(x)\right]^{\frac{\delta-\gamma}{\alpha-\gamma}}   \left[M_{\gamma}f(x)\right]^{1-\frac{\delta-\gamma}{\alpha-\gamma}}\right\Vert_{L^{qs}(B(x_0,\rho),d\mu)}\\
	&\lesssim 
	\Vert f\Vert_{\mathcal{M}_{p,\ell}^{\lambda}(d\nu)}^{\frac{\delta-\gamma}{\alpha-\gamma}} \left\Vert \,\vert M_{\gamma}f{'}\vert ^{(1-\frac{\delta-\gamma}{\alpha-\gamma})}\right\Vert_{L^{q,s}(B(x_0,\rho),d\mu)}\nonumber\\[0.02in]
	&=\Vert f\Vert_{\mathcal{M}_{p,\ell}^{\lambda}(d\nu)}^{\frac{\delta-\gamma}{\alpha-\gamma}} \left\Vert M_{\gamma}f{'}\right\Vert_{L^{q(1-\frac{\delta-\gamma}{\alpha-\gamma}),\,s(1-\frac{\delta-\gamma}{\alpha-\gamma})}(B(x_0,\rho),d\mu)}^{1-\frac{\delta-\gamma}{\alpha-\gamma}},\nonumber
	%&\lesssim\Vert f\Vert_{\mathcal{M}_{p,l}^{\lambda}(d\nu)}^{1-b} \left\Vert M_{\gamma}f{'}\right\Vert_{L^{qb,l}(B(x_0,\rho),d\mu)}^{b},\label{aux1-ineq}
\end{align}
in view of  $\;\Vert \,\vert g\vert^b\Vert_{L^{q,s}}=\Vert g \Vert_{L^{qb,sb}}^b\;$ for $b=1-\frac{\delta-\gamma}{\alpha-\gamma}$. Now, since $\ell<s$ and $b=(1-\frac{\delta-\gamma}{\alpha-\gamma})\in (0,1)$ we can choose  $\gamma$ close to $\delta$ such that $\ell\leq sb$. It follows from Calder\'on's Lemma  \ref{inclusion-1} that 
\begin{align}
\Vert I_{\delta}f{'}\Vert_{L^{q,s}(B(x_0,\rho),d\mu)}\lesssim\Vert f\Vert_{\mathcal{M}_{p,\ell}^{\lambda}(d\nu)}^{1-b} \left\Vert M_{\gamma}f{'}\right\Vert_{L^{qb,\,\ell}(B(x_0,\rho),d\mu)}^{b}.\label{aux1-ineq}
\end{align}
Now from real interpolation (see Lemma \ref{Hunt-thm}) and  trace principle \cite[Theorem 2]{Adams0} in $L^p(d\mu)$ we will show that 
\begin{equation}\label{dor-de-cabeca} 
	\Vert M_{\gamma}f'\Vert_{L^{\overline{p},\ell}(B_\rho, d\mu)}\lesssim \llbracket\mu\rrbracket_{\beta}^{{1}/{\overline{p}}}\Vert f'\Vert_{L^{p,\ell}(\mathbb{R}^n, d\nu)}
\end{equation}
for all $f'\in L^{p,\ell}(d\nu)$ whenever $1<p<\overline{p}=qb={\beta p}/{(n-\gamma p)}$, $0<\beta\leq n$ and $n-\beta<\gamma p<n$. Indeed, let $\theta\in (0,1)$, $p_0<p<p_1$ and $\bar{p}_0<\bar{p}<\bar{p}_1$ be such that 
\begin{equation*}
\frac{1}{p}=\frac{1-\theta}{p_0}+\frac{\theta}{p_1}\quad \text{ and }\quad \frac{1}{\bar{p}}=\frac{1-\theta}{\bar{p}_0}+\frac{\theta}{\bar{p}_1},
\end{equation*}
where $1<p_i< \bar{p}_i=\frac{\beta p_i}{n-\gamma p_i}$, $0<\beta\leq n$ and  $n-\beta<\gamma p_i<n$, $i=0,1$. Hence, from pointwise inequality $M_\gamma f'(x)\lesssim I_{\gamma}\vert f'(x)\vert$ and \cite[Theorem 2]{Adams0}  we have
\begin{equation}
\Vert M_{\gamma}f'\Vert_{L^{\bar{p}_i, \bar{p}_i}(B_\rho, d\mu)}\lesssim \Vert I_{\gamma}f'\Vert_{L^{\bar{p}_i, \bar{p}_i}(B_\rho, d\mu)}\leq \llbracket\mu\rrbracket_{\beta}^{{1}/{\bar{p}_i}}\Vert f'\Vert_{L^{p_i, p_i}(\mathbb{R}^n, d\nu)},\; i=0,1\nonumber
\end{equation}
 provided that the Radon measure $\mu$ satisfies $\llbracket\mu\rrbracket_{\beta}<\infty$. Therefore, thanks to Hunt's Theorem (see Lemma \ref{Hunt-thm}) 
\begin{equation}
\Vert M_{\gamma}f'\Vert_{L^{\overline{p},\ell}(B_\rho, d\mu)}\lesssim  \llbracket\mu\rrbracket_{\beta}^{{(1-\theta)}/{\overline{p}_0}}\llbracket\mu\rrbracket_{\beta}^{{\theta}/{\overline{p}_1}}\Vert f'\Vert_{L^{p,\ell}(\mathbb{R}^n, d\nu)}= \llbracket\mu\rrbracket_{\beta}^{{1}/{\overline{p}}}\Vert f'\Vert_{L^{p,\ell}(d\nu)} \;\text{as }\;1\leq \ell\leq \infty,\nonumber
\end{equation}
where $1<p<\overline{p}={\beta p}/{(n-\gamma p)}$, $0<\beta\leq n$ and $n-\beta<\gamma p<n$, as required. 
%In other words, from \eqref{a-luz2} and condition $q/\lambda_\ast\leq p/\lambda$ and $0<\beta\leq n$ we are able to take  $(n-\beta)<\gamma p<n$ which yields  $p<\overline{p}$. 
Hence, we can inserting \eqref{dor-de-cabeca} into (\ref{aux1-ineq}) to yield 
\begin{align}
\Vert I_{\delta}f{'}\Vert_{L^{q, s}(B(x_0,\rho),d\mu)}&\lesssim\Vert f\Vert_{\mathcal{M}_{p,\ell}^{\lambda}(\mathbb{R}^n,\,d\nu)}^{1-b}\, \llbracket\mu\rrbracket_{\beta}^{b/\overline{p}}\, \left\Vert f'\right\Vert_{L^{p,\ell}(\mathbb{R}^n,d\nu)}^{b}\nonumber\\[0.02in]
&=\Vert f\Vert_{\mathcal{M}_{p,\ell}^{\lambda}(\mathbb{R}^n,\,d\nu)}^{1-b}\,\llbracket\mu\rrbracket_{\beta}^{b/\overline{p}}\, \left\Vert f\right\Vert_{L^{p,\ell}(B(x_0,2\rho),d\nu)}^{b}\nonumber\\[0.02in]
&\lesssim\llbracket\mu\rrbracket_{\beta}^{b/\overline{p}}\,\rho^{\left(\frac{n}{p}-\frac{n}{\lambda}\right)(1-\frac{\delta-\gamma}{\alpha-\gamma})} \Vert f\Vert_{\mathcal{M}_{p,\ell}^{\lambda}(\mathbb{R}^n,\,d\nu)}\nonumber\\[0.02in]
&=\llbracket\mu\rrbracket_{\beta}^{\frac{1}{q}}\,\rho^{\beta\left(\frac{1}{q}-\frac{1}{\lambda_{\ast}}\right)} \Vert f\Vert_{\mathcal{M}_{p,\ell}^{\lambda}(\mathbb{R}^n,\,d\nu)},\label{na-bola}
\end{align}
where the equality \eqref{na-bola} is a consequence of (\ref{a-luz1}) and \eqref{a-luz2} below. Indeed, note that by  $\alpha=n/\lambda$ and $\delta=n/\lambda -\beta/\lambda_{\ast}$, the request 
\begin{align}
	qb=q\left(1-\frac{\delta-\gamma}{\alpha-\gamma}\right)%=q\frac{\alpha-\delta}{\alpha-\gamma}
	=\overline{p}=\frac{\beta p}{n-\gamma p}\label{a-luz1}
\end{align}
is equivalent to 
\begin{align}\label{a-luz2}
	\gamma p\beta\left(1-\frac{q}{\lambda_\ast}\right)=n\beta \left(1-\frac{q}{\lambda_{\ast}}\right)-\beta n\left(1-\frac{p}{\lambda}\right).
\end{align}
Hence, we obtain 
\begin{align*}
{\left(\frac{n}{p}-\frac{n}{\lambda}\right)\left(1-\frac{\delta-\gamma}{\alpha-\gamma}\right)}&\stackrel{\eqref{a-luz1}}{=}\frac{\overline{p}}{q} \left(\frac{n}{p}-\frac{n}{\lambda}\right) \\[0.02in]
&=\frac{\overline{p}}{q}\frac{1}{p\beta}\left[n\beta\left(1-\frac{p}{\lambda}\right)\right]\\[0.02in]
&\stackrel{\eqref{a-luz2}}{=}\frac{1}{q}\frac{1}{n-\gamma  p}\,\left[ n\beta \left(1-\frac{q}{\lambda_{\ast}}\right)-\gamma p\beta\left(1-\frac{q}{\lambda_\ast}\right)\right]\\[0.02in]
&=\beta\left(\frac{1}{q}-\frac{1}{\lambda_{\ast}}\right).
\end{align*}
Note that $(n-\beta)/p<\gamma<\delta<\alpha$ implies $0<b<1$. From estimates \eqref{fora-da-bola} and \eqref{na-bola} we obtain 
\begin{equation}
\rho^{-\beta\left(\frac{1}{q}-\frac{1}{\lambda_{\ast}}\right)} \Vert I_{\delta}f\Vert_{L^{q, s}(\mu\lfloor_{\Omega}(B_\rho))}\lesssim \llbracket\mu\rrbracket_{\beta}^{\frac{1}{q}}\,\Vert f\Vert_{\mathcal{M}_{p, \ell}^{\lambda}(d\nu)},\nonumber\\[0.02in]
\end{equation}
which is the desired continuity of the map $I_\delta:\mathcal{M}_{p,\ell}^{\lambda}(d\nu)\rightarrow\mathcal{M}_{q,s}^{\lambda_{\ast}}(\Omega,d\mu)$.\fin

\subsection{The condition $\llbracket\mu\rrbracket_{\beta}<\infty$ is necessary}\label{only if part}

Let $B(x_0,r)\subset \mathbb{R}^n$ be a ball centered in $x_0$ and with radius $r>0$. Choosing  $f=\chi_{B(x_0,r)}$ when $x\in B(x_0,r)$ we can estimate
\begin{equation*}
(I_{\delta}f)(x)= \int_{\mathbb{R}^n}\vert x-y\vert^{\delta-n}\chi_{B(x_0,r)}(y)d\nu(y)=\int_{\vert y -x_0\vert<r}\vert x-y\vert^{\delta-n}d\nu(y)\gtrsim r^{\delta-n} \nu(B(x_0,r))=Cr^{\delta}
\end{equation*}
thanks to $\vert x-y\vert\leq 2r$ for $y\in B(x_0,r)$. The previous argument implies that the estimate $(I_{\delta}f)^{\ast}(t)\gtrsim r^{\delta}$ hold for $0<t<\mu(B(x_0,r))$. Hence,  using \eqref{almost_norm1}  (see also \eqref{almost_norm}) we obtain  
\begin{align}\label{key_pointwise}
\Vert I_\delta f\Vert_{L^{q, s}(B(x_0,r),d\mu)}\gtrsim r^{\delta}\left(\int_0^{\mu(B(x_0,r))}t^{\frac{s}{q}-1}ds\right)^{\frac{1}{s}}=C\,r^{\frac{n}{\lambda}-\frac{\beta}{\lambda_{\ast}}} [\mu(B(x_0,r)]^{\frac{1}{q}}.
\end{align}
Since  $I_\delta:\mathcal{M}_{p,\ell}^{\lambda}(d\nu)\rightarrow\mathcal{M}_{q, s}^{\lambda_{\ast}}(d\mu)$ is bounded and $\Vert \chi_{B(x_0,r)}\Vert_{\mathcal{M}_{p,\ell}^{\lambda}(\mathbb{R}^n)}=Cr^{n/\lambda}$, then \eqref{key_pointwise} implies that 
\begin{equation*}
r^{\frac{n}{\lambda}}\gtrsim \Vert I_{\delta}f\Vert_{\mathcal{M}_{q, s}^{\lambda_{\ast}}(d\mu)}\gtrsim  r^{\beta\left(\frac{1}{\lambda_{\ast}}-\frac{1}{q}\right)} \Vert I_{\delta} f\Vert_{L^{q, s}(B(x_0,r),d\mu)}\gtrsim r^{\frac{n}{\lambda}-\frac{\beta}{q}} \mu(B(x_0,r))^{\frac{1}{q}}
\end{equation*}
which yields $\mu(B(x_0,r))\lesssim r^{\beta}$ as desired. \fin

\section{Proof of Corollary \ref{Trace}}\label{extension}
The Calder\'on-Stein's extension operator $\mathcal{E}$ on Lipschitz domain $\Omega$ is defined by  $\mathcal{E}f=f$ in $\overline{\Omega}$ and 
\begin{align}
\mathcal{E}f(x)= \int_{1}^{\infty}f(x^{\prime},x_n+s\delta^{\ast}(x))\psi(s)ds  &\text{ on }\; \mathbb{R}^n\backslash \overline{\Omega}\nonumber
\label{extension0}
\end{align}
where $\psi$ is a continuous function on $[1,\infty)$ such that $\psi(s)=O(s^{-N})$ as $s\rightarrow \infty$ for every  $N$, 
\begin{equation}
\int_{1}^{\infty}\psi(s)ds=1\quad \text{ and }\quad \int_{1}^{\infty}s^{k}\psi
(s)ds=0,\text{ for }k=1,2,\cdots\nonumber
\end{equation} 
and $\delta^{\ast}(x)=2c\Delta(x)$ is a $C^{\infty}-$ function comparable to   $\delta(x)=\text{dist}(x,\overline{\Omega})$, see \cite[Theorem 2]{Stein}. On half-space  $\mathbb{R}^n_+$ one has  $\delta^{\ast}(x)=2x_n$ and we have
\begin{equation}
\mathcal{E}f(x^{\prime},x_{n})=\int_{1}^{\infty}f(x^{\prime},(1-2s)x_{n})\psi(s)ds\;\;  \text{ if } \;\;
x_{n}<0
\label{extension1}
\end{equation}
provided that the above integral converges. The proof of  the Lemma \ref{lem-ext1} below is similar to   \cite[Lemma 3.1]{MP1}, we include the proof for reader convenience.

\begin{lemma}\label{lem-ext1} Let $n\geq 2$ and $f\in L^1_{loc}(\mathbb{R}_+^n)$ such that $\nabla f\in \mathcal{M}_{p, d}^{\lambda}(\mathbb{R}_{+}^{n})$ then 
$$
	\Vert\nabla \mathcal{E}f\Vert_{\mathcal{M}^{\lambda}_{p, d}(\mathbb{R}^{n})}\leq C\Vert\nabla 
	f\Vert_{\mathcal{M}^{\lambda}_{p, d}(\mathbb{R}_{+}^{n})}\,
$$
for $1\leq p\leq\lambda<\infty$ and  $d\in[1,\infty]$. 
\end{lemma}

\begin{proof} For each $x'\in \mathbb{R}^{n-1}$ fixed and multi-index $\alpha$, the scaling property 
$
\left\Vert D^{\alpha} f(\gamma \cdot)\right\Vert_{\mathcal{M}_{p, d}^{\lambda}}=\gamma^{\vert\alpha\vert-\frac{n}{\lambda}}\left\Vert f\right\Vert _{\mathcal{M}_{p, d}^{\lambda}}
$ yields 
	\begin{equation}
	\Vert D^{\alpha}f(\cdot, (2s-1)x_n)\Vert_{\mathcal{M}_{p, d}^{\lambda}(\mathbb{R}_{+}^{n})}=(2s-1)^{\vert \alpha\vert-\frac{1}{\lambda}}\Vert D^{\alpha}f(\cdot,x_n)\Vert_{\mathcal{M}_{p, d}^{\lambda}(\mathbb{R}_{+}^{n})}\nonumber.
	\end{equation}
It follows that 
\begin{align*}
\left\Vert \frac{\partial}{\partial x_n}\mathcal{E}f \mathbbm{1}_{\{x_n<0\}}\right\Vert_{\mathcal{M}_{p, d}^{\lambda}(\mathbb{R}^{n})}&=\left\Vert \int_{1}^{\infty}\partial_n f(x^{\prime},(2s-1)x_{n})\psi(s)ds \right\Vert_{\mathcal{M}_{p,d}^{\lambda}(\mathbb{R}^{n}_+)}\\
&\leq \int_{1}^{\infty} (2s-1)\left\Vert \partial_n f(x^{\prime},(2s-1)x_{n}) \right\Vert_{\mathcal{M}_{p,d}^{\lambda}(\mathbb{R}^{n}_+)}\vert \psi(s)\vert ds\\
&\leq \left(\int_{1}^{\infty} (2s-1)^{2-\frac{1}{\lambda}}\vert \psi(s)\vert ds\right)\left\Vert \partial_n f \right\Vert_{\mathcal{M}_{p, d}^{\lambda }(\mathbb{R}^{n}_+)}\\
&\leq C\left\Vert \partial_nf \right\Vert_{\mathcal{M}_{p, d}^{\lambda }(\mathbb{R}^{n}_+)},
\end{align*}
because $\vert \psi(s)\vert \leq C s^{-N}$ for all $N$ implies 
$$\int_{1}^{\infty} (2s-1)^{2-\frac{1}{\lambda}}\vert \psi(s)\vert ds \leq C\int_1^{\infty}(s-1)^{\theta-1}s^{-\theta-(N-\theta)}ds=C \beta(\theta, N-\theta)$$ 
where $\beta(\theta,N-\theta)$ denotes the beta function and $\theta=3-1/\lambda$. Since $\mathcal{E}f=f$ in $\overline{\mathbb{R}^n_+}$, then $\Vert \nabla \mathcal{E}f\mathbbm{1}_{\{x_n\geq  0\}}\Vert_{\mathcal{M}_{p, d}^{\lambda}(\mathbb{R}^{n})}=\Vert \nabla f\Vert_{\mathcal{M}_{p, d}^{\lambda}(\mathbb{R}^{n}_+)}$ and moreover 
\begin{align*}
\left\Vert {\partial_{ x_j}}\mathcal{E}f\mathbbm{1}_{\{x_n<0\}}\right\Vert_{\mathcal{M}_{p, d}^{\lambda}(\mathbb{R}^{n})} \leq \left(\int_1^{\infty} (2s-1)^{1-\frac{1}{\lambda}}\vert\psi(s)\vert ds \right) \left\Vert {\partial_{ x_j}}f\right\Vert_{\mathcal{M}_{p, d}^{\lambda}(\mathbb{R}^{n}_+)}\leq C  \left\Vert {\partial_{ x_j}}f\right\Vert_{\mathcal{M}_{p, d}^{\lambda}(\mathbb{R}^{n}_+)}
\end{align*}
for all $j=1,\cdots,n-1$, as required.
\end{proof}
%Let $\bar{u}\in \mathcal{M}_{pl}^{\lambda}(\mathbb{R}^{n})$ be extension of $u$ and $\bar{f}=I_{\delta}\bar{u}$. 
Thanks to Theorem \ref{measure-restriction-thm} on $\partial\mathbb{R}^n_+$ with $\beta=n-1$, integral representation formula \cite[(3.5)]{Adams} and Lemma \ref{lem-ext1} 
\begin{equation*}
\Vert f(x',0)\Vert_{\mathcal{M}_{q, s}^{\lambda_{\ast}}(\partial\mathbb{R}^n_+)} \leq C \Vert \nabla \mathcal{E}f\Vert_{\mathcal{M}_{p, d}^{\lambda}(\mathbb{R}^{n})}\leq C \Vert
\nabla f\Vert_{\mathcal{M}^{\lambda}_{p, d}(\mathbb{R}_{+}^{n})}
\end{equation*}
as desired.

\section*{Acknowledgment.} The authors would like to thank the anonymous referees for careful reading. The remarks and suggestions pointed by referee's in Theorems 1.1 and 3.5, and Corollaries 1.3 and 1.4 improved our manuscript. This project was supported by Brazilian National Council for Scientific and Technological Development (CNPq-409306/2016).

\end{document}